\documentclass[12pt]{amsart}

\usepackage{extsizes,amsrefs,stmaryrd,amsmath}

\usepackage{etex}

\usepackage{hyperref}
\usepackage{txfonts, amsmath,amstext,amsthm,amscd,amsopn,verbatim,amssymb, amsfonts, mathtools,enumitem}

\usepackage{todonotes}

\usepackage[bbgreekl]{mathbbol}

\usepackage{tikz}

\usetikzlibrary{matrix}
\usetikzlibrary{patterns, shapes}
\usetikzlibrary{arrows}
\usetikzlibrary{calc,3d}
\usetikzlibrary{decorations,decorations.pathmorphing, decorations.pathreplacing}
\usetikzlibrary{through}
\tikzset{ext/.style={circle, draw,inner sep=1pt},int/.style={circle,draw,fill,inner sep=1pt},nil/.style={inner sep=1pt}}
\tikzset{exte/.style={circle, draw,inner sep=3pt},inte/.style={circle,draw,fill,inner sep=3pt}}
\tikzset{diagram/.style={matrix of math nodes, row sep=3em, column sep=2.5em, text height=1.5ex, text depth=0.25ex}}
\tikzset{diagram2/.style={matrix of math nodes, row sep=0.5em, column sep=0.5em, text height=1.5ex, text depth=0.25ex}}

\usepackage{tikz-cd}

\theoremstyle{plain}

\newtheorem{thm}{Theorem}[section]

\newtheorem{prop}[thm]{Proposition}
\newtheorem{cor}[thm]{Corollary}
\newtheorem{lemma}[thm]{Lemma}

\theoremstyle{definition}

\newtheorem{rem}[thm]{Remark}

\newcommand{\ad}{{\text{ad}}}

\newcommand{\p}{\partial}

\newcommand{\R}{{\mathbb{R}}}

\newcommand{\Z}{{\mathbb{Z}}}

\newcommand{\Q}{{\mathbb{Q}}}

\newcommand{\HGC}{{\mathrm{HGC}}}

\newcommand{\bpm}{\begin{pmatrix}}
\newcommand{\epm}{\end{pmatrix}}

\newcommand{\GC}{\mathrm{GC}}

\newcommand{\MC}{\mathsf{MC}}

\DeclareMathOperator{\rk}{rk}

\DeclareMathOperator{\Emb}{Emb}
\DeclareMathOperator{\Imm}{Imm}
\DeclareMathOperator{\Embbar}{\overline{Emb}}

\DeclareMathOperator{\Diff}{Diff}

\newcommand{\beq}[1]{\begin{equation}\label{#1}}
\newcommand{\eeq}{\end{equation}}

\newcommand{\hofiber}{\mathrm{hofiber}}

\newcommand{\SO}{\mathrm{SO}}

\newcommand{\ar}{\mathrm{ar}}

\begin{document}
\title{On the  homotopy type of the  spaces of spherical  knots in $\R^n$}
%

\author{Victor Turchin}
\address{Department of Mathematics\\
Kansas State University\\
Manhattan, KS 66506, USA}
\email{turchin@ksu.edu}

\author{Thomas Willwacher}
\address{Department of Mathematics \\ ETH Zurich \\
R\"amistrasse 101 \\
8092 Zurich, Switzerland}
\email{thomas.willwacher@math.ethz.ch}


\thanks{V.T. has benefited from a visiting position of the Labex CEMPI (ANR-11-LABX-0007-01) at the Universit\'e de Lille and a visiting  position at the Max Planck Institute for Mathematics in Bonn for the achievement of this work.  V.T.  has also been 
partially supported by the Simons Foundation grant, award ID:~519474. 
T.W. has been partially supported by the NCCR SwissMAP funded by the Swiss National Science Foundation, and the ERC starting grant GRAPHCPX (678156)}


\begin{abstract}
We study the spaces of embeddings $S^m\hookrightarrow \R^n$ and those of long embeddings $\R^m\hookrightarrow\R^n$, i.e. embeddings of a fixed behavior  
outside a compact set.   More precisely we look at 
the homotopy fiber of the inclusion of these spaces to the spaces of immersions. We find a natural fiber sequence relating  these spaces.
We also compare the  $L_\infty$-algebras of diagrams that encode their rational homotopy type, when the codimension
$n-m\geq 3$.

\end{abstract}

\maketitle


\sloppy
\section{Introduction}
In this paper we study a relation between the following two spaces:
\begin{gather}
\Embbar(S^m,\R^n):=\hofiber\left(\Emb(S^m,\R^n)\to\Imm(S^m,\R^n)\right); \label{eq:hofib_sphere} \\
\Embbar_\p(\R^m,\R^n):=\hofiber\left(\Emb_\p(\R^m,\R^n)\to \Imm_\p(\R^m,\R^n)\right), \label{eq:hofib_long}
\end{gather}
where $\Emb(-,-)$ and $\Imm(-,-)$ always refer to spaces of smooth embeddings and immersions, respectively. The homotopy fiber is taken over
the inclusions   $i_1\colon S^m\subset \R^{m+1}\times 0^{n-m-1}\subset \R^n$ and $i_2\colon\R^m=\R^m\times 0^{n-m}\subset\R^n$. The subscript
$\partial$ means that the embeddings and immersions must coincide with the inclusion $i_2\colon\R^m\subset\R^n$ outside a compact subset of $\R^m$. The spaces~\eqref{eq:hofib_sphere} and~\eqref{eq:hofib_long}  are called
{\it spaces of embeddings modulo immersions}.

The spaces $\Embbar_\p(\R^m,\R^n)$ have been objects of active study~\cite{Turchin2,Turchin3,WBdB2,Budney1,Budney2,Budney3,BudneyCohen,
DucT,FTW,Sakai,SakaiWatanabe}. They were shown to be $E_{m+1}$-algebras~\cite{Budney1,Budney3,Sakai} equivalent to $(m+1)$-loop spaces~\cite{WBdB2,DucT,Sakai} when $n-m\geq 3$.

To compare their homotopy type to that of $\Embbar(S^m,\R^n)$ let us begin with a few observations.
Given an embedding $\psi\in\Embbar(S^m,\R^n)$, we can define an inclusion 
\beq{eq:Psi}
\Embbar_\p(\R^m,\R^n)\hookrightarrow \Embbar(S^m,\R^n).
\eeq
The idea of this map is to perturb $\psi$ near some point $p\in S^m$. By  standard fibration and transversality arguments it is easy to show that for $n-m\geq 3$, $\pi_*\Emb_\p(\R^m,\R^n)\simeq\pi_*\Emb(S^m,\R^n)$ and $\pi_*\Imm_\p(\R^m,\R^n)\simeq\pi_*\Imm(S^m,\R^n)$,
for $*\leq 1$. This implies that the inclusion~\eqref{eq:Psi} induces a bijection of the sets of connected components $\pi_0\Embbar_\p(\R^m,\R^n)\simeq\pi_0\Embbar(S^m,\R^n)$. It is also not hard to show that the inclusion~\eqref{eq:Psi} can be enhanced to an $\Embbar_\p(\R^m,\R^n)$-action on
$\Embbar(S^m,\R^n)$.\footnote{To carefully define this action the spaces $\Embbar_\p(\R^m,\R^n)$ and $\Embbar(S^m,\R^n)$ need to be replaced by the homotopy equivalent spaces  $\Embbar{}^{fr}_\p(\R^m,\R^n)$, $\Embbar{}^{fr}(S^m,\R^n)$ of framed embeddings modulo framed immersions.}  Our main result now states that the homotopy quotient of $\Embbar(S^m,\R^n)$ by the action of $\Embbar_\p(\R^m,\R^n)$ is the sphere~$S^{n{-}m{-}1}$, or equivalently that $\Embbar(S^m,\R^n)$ is homotopy equivalent to a principal $\Embbar_\p(\R^m,\R^n)$-bundle on the sphere.


\begin{thm}\label{thm:emb_long}
    For $n-m\geq 3$, one has an equivalence
    \beq{equ:emb_long_intro}
    \Embbar(S^m,\R^n)\simeq\hofiber\left(S^{n-m-1}\xrightarrow{} 
    B\Embbar_\p(\R^m,\R^n)\right).
    \eeq
    This in particular  implies that all connected components of $\Embbar(S^m,\R^n)$ have the same homotopy type and that $\pi_0\Embbar(S^m,\R^n)=\pi_0\Embbar_\p(\R^m,\R^n)$.
\end{thm}
For an explicit definition of the classifying map $S^{n-m-1}\xrightarrow{} 
B\Embbar_\p(\R^m,\R^n)$ appearing in this Theorem we refer to Section \ref{s:emb_long}. 


As a consequence we will in particular be able to express the rational homotopy types of $\Embbar_\p(\R^m,\R^n)$ and $\Embbar(S^m,\R^n)$ through each other, see Corollaries \ref{cor:Qranks}, \ref{cor:rht} below. It is furthermore well known that the rational homotopy type of  $\Embbar_\p(\R^m,\R^n)$, $n-m\geq 3$, may be expressed through hairy graph-complexes.
More precisely, in~\cite{Turchin3} hairy graph-complexes were introduced,  denoted by $\HGC_{\bar A_m,n}$ in this paper,  which were proved
to compute the rational homotopy groups
\beq{eq:hgc_bar}
H_*(\HGC_{\bar A_m,n})\simeq \Q\otimes \pi_*\Embbar_\p(\R^m,\R^n)
\eeq
for $n\geq 2m+2$. The  paper~\cite{FTW} determined the rational homotopy type of the $(m+1)$-st delooping of $\Embbar_\p(\R^m,\R^n)$, $n-m\geq 3$. 
In particular \cite[Theorem~15 and Remark~19]{FTW} improved the equality~\eqref{eq:hgc_bar} to the range $n-m\geq 3$. In that range, the
space $\Embbar_\p(\R^m,\R^n)$ can be disconnected, but since it is an $(m+1)$-loop space its set of connected components forms an abelian group (in fact finitely generated). The cited theorem proves the isomorphism~\eqref{eq:hgc_bar} in degree zero as well. Note, however, that the graph-complex $\HGC_{\bar A_m,n}$
can have non-trivial homology in negative degrees, that has to be ignored.\footnote{In fact the non-positive degree homology $H_{\leq 0}(\HGC_{\bar A_m,n})$, that includes the negative degree and degree zero, is at most one-dimensional for $n-m\geq 3$.}

Recently, in~\cite{FTW2} a more general method has been developed by B.~Fresse and the authors to study the rational homotopy type of (connected components of)  embeddings modulo immersions spaces
$\Embbar(L,\R^n)$ and $\Embbar_\p(L,\R^n)$, where $L$ is either a compact submanifold of $\R^{m+1}$  with components of possibly different dimensions, or a closed submanifold whose unbounded connected components  coincide with affine subspaces of $\R^{m+1}$ outside a ball of 
some radius $R$. The main result of~\cite{FTW2} provides  $L_\infty$-algebras of diagrams that express the rational type of such spaces\footnote{This result  uses 
the general theory of Postnikov decompositions of (modules over) reduced operads (i.e. operads whose arity zero component is reduced to a point), which is 
a work in progress
by M.~Mienn\'e~\cite{MienneMemoir}. The theory of Postnikov decompositions
of operads with the empty arity zero component appeared  in Mienn\'e's thesis~\cite{MienneThesis}.}.
In particular, for the first non-trivial case of $L=S^{m}$ the corresponding $L_\infty$-algebra is a hairy graph-complex denoted by
$\HGC_{A_m,n}$.

On the rational homotopy level the comparison of the embedding spaces $\Embbar_\p(\R^m,\R^n)$ and $\Embbar(S^m,\R^n)$ pursued in this paper hence translates into a comparison of the graph-complexes $\HGC_{\bar A_m,n}$ and $\HGC_{A_m,n}$. We shall explain in Section~\ref{s:loneg_emb_graphs} how the relation between the spaces of Theorem \ref{thm:emb_long} can be seen directly (and independently) on the graph-complexes, at least rationally. In fact, this is how we initially discovered our Theorem \ref{thm:emb_long}.
Computations from Section~\ref{s:loneg_emb_graphs} could be useful in further pursuing the graph-complex approach
from~\cite{FTW2} applying it to other types of manifolds.

In the last Section~\ref{s4} we study the case of codimension $n-m\leq 2$.   The main statement of Theorem~\ref{thm:emb_long}
fails in all cases when $n-m=2$, but still holds when $n-m=1$ provided $n=3$ or~$7$.



\section{Spherical and long embeddings}\label{s:emb_long}
In this section  we describe how the homotopy type of $\Embbar(S^m,\R^n)$ is compared to  that of $\Embbar_\p(\R^m,\R^n)$, and in particular prove Theorem \ref{thm:emb_long}. Throughout Sections~\ref{s:emb_long} and~\ref{s:loneg_emb_graphs}  we assume $n-m\geq 3$.

\subsection{Proof of Theorem \ref{thm:emb_long}}


The second statement of the theorem holds because the sphere $S^{n-m-1}$ is simply connected.
By the Smale-Hirsch theorem \cite{Hirsch,Smale}, $\Imm_\p(\R^m,\R^n)\simeq \Omega^m V_m(\R^n)$, where $V_m(\R^n)=\SO(n)/\SO(n{-}m)$ is the Stiefel manifold
of orthogonal $m$-frames in $\R^n$. Thus
\beq{equ:embbar_def}
\Embbar_\p(\R^m,\R^n)\simeq \hofiber\Bigl( \Emb_\p(\R^m,\R^n) \xrightarrow{D} \Omega^mV_m(\R^n)\Bigr).
\eeq
The space $\Embbar_\p(\R^m,\R^n)$ is an $(m+1)$-loop space~\cite{WBdB2,DucT,Sakai}. We denote by $B\Embbar_\p(\R^m,\R^n)$ its classifying space
and by $g$ the map
\beq{equ:map_g}
g\colon\Omega_*^mV_m(\R^n)\simeq B\Omega^{m+1}V_m(\R^n)\to  B\Embbar_\p(\R^m,\R^n)
\eeq
(where $\Omega_*$ stays for the loop space  component of the constant map) obtained by applying the classifying space functor $B$ to the inclusion $\Omega^{m+1}V_m(\R^n)\to \Embbar_\p(\R^m,\R^n)$. 

Consider also the map
\beq{equ:map_h}
h\colon S^{n-m-1}\to \Omega^mV_m(\R^n)
\eeq
adjoint to the composition
\beq{equ:map_h0}
\Sigma^mS^{n-m-1} =S^{n-1}\xrightarrow{h_0} \SO(n)\to \SO(n)/\SO(n{-}m)=V_m(\R^n),
\eeq
where $h_0$ is the transition map for the tangent bundle of $S^n=D^n_+\cup_{S^{n-1}}D^n_-$ relating trivializations over the upper and lower discs
$D^n_+$ and $D^n_-$. Note that since we assume $n-m\geq 3$,
the sphere $S^{n-m-1}$ is connected and $h(S^{n-m-1})\subset \Omega_*^mV_m(\R^n)$.

To show Theorem \ref{thm:emb_long} we will check explicitly that for $n-m\geq 3$ one has an equivalence
\beq{equ:emb_long}
\Embbar(S^m,\R^n)\simeq\hofiber\left(S^{n-m-1}\xrightarrow{g\circ h}B\Embbar_\p(\R^m,\R^n)\right).
\eeq

In other words, the space $\Embbar(S^m,\R^n)$ is equivalent to a principal $\Embbar_\p(\R^m,\R^n)$-bundle over $S^{n-m-1}$ with the structure subgroup
$\Omega^{m+1}V_m(\R^n)\subset\Embbar_\p(\R^m,\R^n)$. 
%

The equivalence \eqref{equ:emb_long} and hence Theorem~\ref{thm:emb_long} can be shown using the following two propositions. 

\begin{prop}\label{prop:hofib1}
For $n-m\geq 3$, one has an equivalence
\[
\Embbar(S^m,\R^n)\simeq\hofiber\left(\Emb_\p(\R^m,\R^n)\times S^{n-m-1}\xrightarrow{m\circ(D\times h)}\Omega^mV_m(\R^n)\right),
\]
where $m\colon \Omega^mV_m(\R^n)\times\Omega^mV_m(\R^n)\to \Omega^mV_m(\R^n)$ is a loop product. 
\end{prop}
To recall $D$ and $h$ denote the maps from~\eqref{equ:embbar_def} and~\eqref{equ:map_h}.
Proposition~\ref{prop:hofib1} is related to and in fact is a consequence of   Budney-Cohen's 
\cite[Proposition 4.4]{BudneyCohen}.
For completeness of exposition we present its full proof below.

\begin{prop}\label{prop:hofib2}
Let $Y\xrightarrow{f}X$ be a map of pointed spaces, $Z\xrightarrow{h}\Omega X$ be any map. Let also 
$\Omega X\xrightarrow{i}\hofiber(Y\xrightarrow{f}X)$ denote the natural inclusion and $m\colon\Omega X\times\Omega X\to
\Omega X$ denote the loop product. One has an equivalence
\beq{equ:hofib2}
\hofiber\left(Z\xrightarrow{i\circ h}\hofiber(Y\xrightarrow{f}X)\right)\simeq
\hofiber\left(\Omega Y\times Z\xrightarrow{m\circ(\Omega f\times h)}\Omega X\right).
\eeq
\end{prop}

\begin{proof}[Proof of Theorem~\ref{thm:emb_long}.]
We apply Proposition~\ref{prop:hofib1} and Proposition~\ref{prop:hofib2} to
the case $Y\xrightarrow{f}X$ being $B\Emb_\p(\R^m,\R^n)\xrightarrow{BD}\Omega_*^{m-1}V_m(\R^n)$, and 
$Z\xrightarrow{h}\Omega X$ being $S^{n-m-1}\xrightarrow{h}\Omega^mV_m(\R^n)$.
One has, $\Omega B\Emb_\p(\R^m,\R^n)\simeq \Emb_\p(\R^m,\R^n)$,
because $\pi_0 \Emb_\p(\R^m,\R^n)$, $n-m\geq 3$, is a group~\cite{Haefliger1,Haefliger2}. (Explicit
deloopings of  $\Emb_\p(\R^m,\R^n)$, $n-m\geq 3$, were obtained in~\cite{WBdB2,DucT,Sakai}.) Note that
\beq{eq:hof_discon}
\hofiber\left(B\Emb_\p(\R^m,\R^n)\xrightarrow{BD}\Omega_*^{m-1}V_m(\R^n)\right) \simeq \Omega^mV_m(\R^n)
\sslash \Emb_\p(\R^m,\R^n).
\eeq
The sphere $S^{n-m-1}$ is connected and each connected component of~\eqref{eq:hof_discon} 
 is equivalent to $B\Embbar_\p(\R^m,\R^n)$, which  immediately yields~\eqref{equ:emb_long}.

\end{proof}

\begin{proof}[Proof of Proposition~\ref{prop:hofib1}.]
Denote by $\Emb_*(S^m,\R^n)$ and $\Imm_*(S^m,\R^n)$ the spaces of embeddings and immersions, respectively, with a fixed behavior near
the basepoint $*\in S^m$. One can easily see that the space
\[
\Embbar_*(S^m,\R^n):=\hofiber\left(\Emb_*(S^m,\R^n)\xhookrightarrow{I}\Imm_*(S^m,\R^n)\right)
\]
is weakly equivalent to $\Embbar(S^m,\R^n)$. Moreover, we claim that $\Emb_*(S^m,\R^n)\simeq \Emb_\p(\R^m,\R^n)\times S^{n-m-1}$ and
$\Imm_*(S^m,\R^n)\simeq\Omega^mV_m(\R^n)$ with the map $I$ of the homotopy type of $m\circ(D\times h)$.


We decompose $S^m=D^m_+\cup_{S^{m-1}} D^m_-$, where $D^m_-$ is a small closed disc neighborhood of the basepoint $*\in S^m$, and $D^m_+$ is its
complementary disc.  We identify $\R^n=S^n\setminus\{N\}$ as a sphere without its north pole. Similarly we decompose $S^n\setminus\{N\}=
\left(D^n_+\setminus\{N\}\right)\cup_{S^{n-1}}D^n_-$. One has
\beq{equ:emb*}
\Emb_*(S^m,\R^n)\cong\Emb_\p(D^m_+,D^n_+\setminus\{N\}),
\eeq
\beq{equ:imm*}
\Imm_*(S^m,\R^n)\cong\Imm_\p(D^m_+,S^n\setminus\{N\})\simeq\Omega^mV_m(\R^n).
\eeq
The last equivalence in~\eqref{equ:imm*} is by the Smale-Hirsch theorem, as the target manifold $S^n\setminus\{N\}=\R^n$ is contractible. 

\begin{rem}\label{rem:framing}
The transition map between the coordinate framing on $\R^n=S^n\setminus\{N\}$ and the local coordinates framing near~$N$, when restricted on a small $(n-1)$-sphere around $N$, is given by the map $h_0$ from equation~\eqref{equ:map_h0}.
\end{rem}

Consider the space $\Emb_\p(\R^m\sqcup\{*\},\R^n)$, where $\R^m\sqcup\{*\}$ is given the disjoint union topology. Below we define maps
\beq{equ:ABC}
\begin{tikzcd}
\Emb_\p(\R^m,\R^n)\times S^{n-m-1}\ar{d}{A} \ar{dr}{B} \\ 
\Emb_\p(\R^m,\R^n\setminus\{0\})\ar{r}{C} & \Emb_\p(\R^m\sqcup\{*\},\R^n)
\end{tikzcd}\, .
\eeq

The map $C$ is the inclusion sending $f\mapsto\tilde f$, where 
$$
\tilde f(x)=
\begin{cases}
f(x),&x\in\R^m;\\
0,&x=*.
\end{cases}
$$

By $S^{n-m-1}$ we understand the unit sphere in $\R^{n-m}$. Map $B$ sends a pair $(f,v)$ to $\tilde f$, such that $\tilde f(*)=0^m\times v$ and
$\tilde f|_{\R^m}$ is supported in the unit ball with center ${-}3\times 0^{m-1}$ and sending this ball inside the unit ball centered at ${-3}\times 0^{n-1}$. 
We use here the homeomorphism $\Emb_\p(\R^m,\R^n)\cong \Emb_\p(D^m,D^n)$ induced by a diffeomorphism
between $\R^n$ and the interior $int(D^n)$ of $D^n$, that sends  $\R^m$ to $int(D^m)$.

Finally we define $A$. Let $\rho\colon\R^m\to [0,1]$ be a smooth bump function  supported in the unit disc $D^m$. The map $A$ sends $(f,v)$ to 
$\tilde f\colon\R^m\hookrightarrow\R^n$ supported in the union of two unit discs with centers ${-}3\times 0^{m-1}$ and $0^m$. Inside the first disc $\tilde f$ 
is defined in the same way as in the case of map~$B$, while inside the second disc $\tilde f(x)=(x,-\rho(x)v)$.

\begin{lemma}\label{l:ABC}
For $n-m\geq 3$, all three maps $A$, $B$, $C$ in~\eqref{equ:ABC} are weak homotopy equivalences. Moreover, $B$ is homotopic to $C\circ A$.
\end{lemma}
\begin{proof}
Consider two fibrations 
\begin{align*}
\pi_1\colon\Emb_\p(\R^m\sqcup\{*\},\R^n)&\to\R^n;\\
\pi_2\colon\Emb_\p(\R^m\sqcup\{*\},\R^n)&\to\Emb_\p(\R^m,\R^n)
\end{align*}
obtained by restricting embeddings to one of the components: $\{*\}$ or $\R^m$. 

Since the target of $\pi_1$ is contractible, the inclusion of the fiber in the total space is an equivalence,
 implying that $C$ is a weak equivalence.

The map $B$ is a morphism of fiber bundles over $\Emb_\p(\R^m,\R^n)$. By applying the Alexander duality and the fact that both $S^{n-m-1}$ and the fiber of $\pi_2$
(the complement of a knot) are simply connected, we get that $B$ induces an equivalence of fibers, therefore is an equivalence of total spaces. 

It is obvious that $B\simeq C\circ A$. By the two out of three property, $A$ is also a weak equivalence.
\end{proof}

To finish the proof of Proposition~\ref{prop:hofib1}, one has to show that the composition
\begin{multline}\label{equ:tildeA}
J\colon\Emb_\p(\R^m,\R^n)\times S^{n-m-1}\xrightarrow[A]{\simeq}
\Emb_\p(\R^m,\R^n\setminus\{0\})\xrightarrow{\cong}\\ \Emb_\p(D^m_+,D^n_+\setminus\{N\})\to
\Imm_\p(D^m_+,S^n\setminus\{N\})\xrightarrow{\simeq}\Omega^mV_m(\R^n)
\end{multline}
is homotopic to $m\circ(D\times h)$. It is obvious that $J$ restricted to the first factor $\Emb_\p(\R^m,\R^n)$ is homotopic to $D$.
It follows from Remark~\ref{rem:framing}, that $J$ restricted on the second factor $S^{n-m-1}$ is homotopic to $h$. 
Also by construction $J$ is a concatenation of the loop obtained from the first factor with the loop obtained from the second factor,
which is exactly what the formula $m\circ(D\times h)$ means.
\end{proof}

\begin{proof}[Proof of Proposition~\ref{prop:hofib2}.]
Recall the standard construction of the homotopy fiber of a map $Y\xrightarrow{f}X$. It is the space of pairs $(y,x)$, where $y\in Y$ and
$x\colon[0,1]\to X$, such that $x(0)=*$ and $x(1)=f(y)$. When this construction is applied, both spaces in~\eqref{equ:hofib2} are
homeomorphic to the space of triples $(z,y,x)$, where $z\in Z$, $y\in\Omega Y$, $x\colon D^2\to X$, such that $x|_{\p D^2}$ is the loop 
$m\left(h(z),(\Omega f)(y)\right)$. 
\end{proof}

\subsection{Corollaries for the rational homotopy types}

The rational homotopy $\pi_*^\Q S^{n-m-1}$ is spanned by the spherical class $\iota\in\pi_{n-m-1}^\Q S^{n-m-1}$  and the Hopf class
$[\iota,\iota]\in\pi_{2n-2m-3}^\Q S^{n-m-1}$, which is non-zero only if $n-m$ is odd.
The induced map in the rational homotopy $h_*\colon\pi^\Q_*S^{n-m-1}\to\pi_*^\Q\Omega^mV_m(\R^n)$, sends the spherical class $\iota$ to the
$\SO(n)$ Euler class, if $n$ is even, and sends it to zero if $n$ is odd. The Hopf class $[\iota,\iota]$ of $S^{n-m-1}$ is sent to zero, because
the rational homotopy of any loop space is an abelian Lie algebra. Recall also that the induced map $g_*\colon\pi_*^\Q\Omega^mV_m(\R^n)
\to \pi_*^\Q B\Embbar_\p(\R^m,\R^n)$, sends the $\SO(n)$ Euler class to the graph-cycle
\[
D=\begin{tikzpicture}[baseline=-.65ex,every loop/.style={}]
\node (v) at (0,0) {$\omega$};
\node [int] (w) at (1,0) {};
\draw (v) edge (w) (w) edge[loop] (w);
\end{tikzpicture}
\]
in $\HGC_{\bar A_m,n}$, see~\cite{Turchin3,FTW,KhorWillw}, which is non-zero only if $n$ is even.  

Together with Theorem~\ref{thm:emb_long}, the computations above  immediately imply:

\begin{cor}\label{cor:Qranks}
For $n-m\geq 3$, one has $\rk\pi_i^\Q\Embbar(S^m,\R^n)=\rk\pi_i^\Q\Embbar_\p(\R^m,\R^n)$, except
\begin{itemize}
\item for $n$ even,  $\rk\pi_{n-m-2}^\Q\Embbar(S^m,\R^n)=\rk\pi_{n-m-2}^\Q\Embbar_\p(\R^m,\R^n)-1$;
\item for $n$ odd,  $\rk\pi_{n-m-1}^\Q\Embbar(S^m,\R^n)=\rk\pi_{n-m-1}^\Q\Embbar_\p(\R^m,\R^n)+1$;
\item for $n-m$ odd,  $\rk\pi_{2n-2m-3}^\Q\Embbar(S^m,\R^n)=\rk\pi_{2n-2m-3}^\Q\Embbar_\p(\R^m,\R^n)+1$.
\end{itemize}
\end{cor}

 (It follows from Theorem~\ref{thm:emb_long} that 
$\pi_1\Embbar(S^m,\R^n)$ is a quotient group of $\pi_1\Embbar_\p(\R^m,\R^n)$ and therefore  is abelian.) 

Any map from a suspension to an $H$-space is rationally coformal (and also formal). For $n$ odd, the induced map in rational homotopy $(g\circ h)_*$ is zero, and for $n$ even it is
non-zero only on the spherical class $\iota$. This immediately determines the rational homotopy type of $\Embbar(S^m,\R^n)$, $n-m\geq 3$.

\begin{cor}\label{cor:rht}
For $n-m\geq 3$,
\begin{itemize}
\item if $n-m$  or $n$ is even, each component of $\Embbar(S^m,\R^n)$ is rationally equivalent to a product of $K(\Q,j)$'s, in other words, it is coformal with an
abelian Quillen model;
\item if $n$ is odd, $\Embbar(S^m,\R^n)\simeq_\Q \Embbar_\p(\R^m,\R^n)\times S^{n-m-1}$.
\end{itemize}
\end{cor}

Only in the case  $n$ odd and $m$ even, the space $\Embbar(S^m,\R^n)$ is not rationally abelian. However, the failure of being non-abelian is only in the rational
factor~$S^{n-m-1}$.

\section{Comparing graph-complexes}\label{s:loneg_emb_graphs}
As described in the introduction the rational homotopy types of both spaces $\Embbar_\p(\R^m,\R^n)$ and $\Embbar(S^m,\R^n)$, $n-m\geq 3$, have known expressions through graph-complexes. The purpose of this section is to illustrate how Theorem~\ref{thm:emb_long} and in particular Corollaries~\ref{cor:Qranks} and~\ref{cor:rht} manifest themselves combinatorially on the graph-complex level.
We shall proceed without using Theorem~\ref{thm:emb_long} directly, but rather by providing independent arguments, thus essentially re-proving (parts of) the theorem rationally.

We will use the notion of (complete) $L_\infty$-algebras and their Maurer-Cartan spaces. 
We adopt Whitehead's grading conventions in which the bracket, higher brackets, and differential of an $L_\infty$-algebra are all of degree $-1$.
We refer the reader to \cite[section 2]{DolRog} for a comprehensive but careful recollection, using the same grading conventions.

\subsection{Hairy graph-complexes}\label{ss:HGC}
In this subsection we describe graph-complexes $\HGC_{\bar A_m,n}$, $\HGC_{A_m,n}$ and their $L_\infty$-algebra structures that
express the rational homotopy type, respectively, of $\Embbar_\p(\R^m,\R^n)$ and $\Embbar(S^m,\R^n)$, $n-m\geq 3$. Here $\bar A_m$ denotes
the  reduced cohomology algebra $\tilde H^*(S^m,\Q)$, and $A_m$ denotes the cohomology algebra $H^*(S^m,\Q)$. The former is spanned by a single element
$\omega$ of degree $m$, while the latter is spanned by $1$ and $\omega$. With Whitehead's grading conventions
\begin{align*}
H_*(\HGC_{\bar A_m,n})=&\Q\otimes\pi_*\Embbar_\p(\R^m,\R^n),\\
H_*(\HGC_{ A_m,n})=&\Q\otimes\pi_*\Embbar(S^m,\R^n),
\end{align*}
and the bracket in graph-complexes corresponds to the Whitehead bracket in the rational homotopy. Note that the latter one is almost always zero according
to Corollary~\ref{cor:rht}.

The graph-complexes are spanned by finite connected graphs with two types of vertices: external ones of valence one (called {\it hairs}) and internal ones
of valence $\geq 3$. Every external vertex is labeled by $\omega$ in case of $\HGC_{\bar A_m,n}$, and either by $\omega$ or by~$1$ in case of $\HGC_{A_m,n}$.
Double edges and tadpoles (edges connecting a vertex to itself) are allowed. Such graphs are required to have at least one hair. Let $E$, $V$, $H$ denote, respectively,
the sets of edges, internal vertices, and $\omega$-hairs of a graph~$\Gamma$. The degree of such graph is 
\[
(n-1)\# E-n\# V-m\# H.
\]
For example the degree of the graph
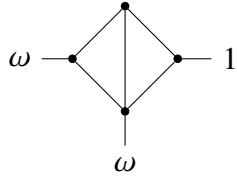
\[
\begin{tikzpicture}[scale=.7,baseline=-.65ex]
\node[int] (v1) at (-1,0){};
\node[int] (v2) at (0,1){};
\node[int] (v3) at (1,0){};
\node[int] (v4) at (0,-1){};
\node (w1) at (-2,0) {$\omega$};
\node (w2) at (2,0) {$1$};
\node (w3) at (0,-2) {$\omega$};
\draw (v1)  edge (v2) edge (v4) edge (w1) (v2) edge (v4) (v3) edge (v2) edge (v4) (v4) edge (w3) (v3) edge (w2);
\end{tikzpicture}
\]
is $4n-2m-8$. Note that the edges at the hairs we also count as edges, so that the diagram above has 8 edges.
By an {\it orientation} of $\Gamma$ we understand an orientation of its edges and a linear order of its {\it orientation set} $E\cup V\cup H$. Changing orientation of an edge gives the sign $(-1)^n$. Changing the order of the orientation set brings in the Koszul sign of permutation, where edges are assigned degree $n-1$, internal vertices
are assigned  degree $-n$, and $\omega$-hairs are assigned degree $-m$.

The differential on $\HGC_{\bar A_m,n}$ is denoted by $\delta_{split}$: it acts by splitting the vertices into two:
\begin{align}\label{equ:deltasplit}
\delta_{split} \Gamma &= \sum_{v \text{ vertex} }  \pm 
\Gamma\text{ split $v$} 
&
\begin{tikzpicture}[baseline=-.65ex]
\node[int] (v) at (0,0) {};
\draw (v) edge +(-.3,-.3)  edge +(-.3,0) edge +(-.3,.3) edge +(.3,-.3)  edge +(.3,0) edge +(.3,.3);
\end{tikzpicture}
&\mapsto
\sum
\begin{tikzpicture}[baseline=-.65ex]
\node[int] (v) at (0,0) {};
\node[int] (w) at (0.5,0) {};
\draw (v) edge (w) (v) edge +(-.3,-.3)  edge +(-.3,0) edge +(-.3,.3)
 (w) edge +(.3,-.3)  edge +(.3,0) edge +(.3,.3);
\end{tikzpicture}
\end{align}
The differential on $\HGC_{A_m,n}$ is $\delta=\delta_{split}+\delta_{join}$, where $\delta_{split}$ is defined by~\eqref{equ:deltasplit}, while $\delta_{join}$
joins a subset of at least two hairs into one hair, multiplying the decorations, schematically:
\begin{align}\label{equ:deltajoin}
 \delta_{join} 
\begin{tikzpicture}[baseline=-.8ex]
\node[draw,circle] (v) at (0,.3) {$\Gamma$};
\node (w1) at (-.7,-.5) {$a_1$};
\node (w2) at (-.25,-.5) {$a_2$};
\node (w3) at (.25,-.5) {$\dots$};
\node (w4) at (.7,-.5) {$a_k$};
\draw (v) edge (w1) edge (w2) edge (w3) edge (w4);
\end{tikzpicture} 
= 
\sum_{\substack{S\subset {\rm hairs} \\ |S|\geq 2 }} \pm 
\begin{tikzpicture}[baseline=-.8ex]
\node[draw,circle] (v) at (0,.3) {$\Gamma$};
\node (w1) at (-.7,-.5) {$a_1$};
\node (w2) at (-.25,-.5) {$\dots$};
\node[int] (i) at (.4,-.5) {};
\node (w4) at (.4,-1.3) {$\scriptstyle \prod_{j\in S}a_j$};
\draw (v) edge (w1) edge (w2) edge[bend left] (i) edge (i) edge[bend right] (i) (w4) edge (i);
\end{tikzpicture} \, .
\end{align}
Clearly, a summand in~\eqref{equ:deltajoin} is non-zero only if $S$ contains at most one $\omega$-hair. For the signs,
note that each graph $\Gamma'$ in the sums $\delta_{split}\Gamma$ and $\delta_{join}\Gamma$ has exactly one more
vertex and one more edge than the initial graph~$\Gamma$. So, to obtain an (ordered) orientation set of $\Gamma'$, we just add to that  of $\Gamma$ the new vertex and new edge as the first and second elements. The new edge of $\Gamma'$ is oriented towards
its new vertex. In case of $\delta_{split}$ there are two choices which vertex is considered as a new one, but the two
 resulting orientations are equivalent. With this convention, all the signs in~\eqref{equ:deltasplit} and~\eqref{equ:deltajoin}  are positive.

The $r$-th $L_\infty$-operation $\ell_r(\Gamma_1,\ldots,\Gamma_r)$, $r\geq 2$, is zero for $\HGC_{\bar A_m,n}$ and is defined similarly to
$\delta_{join}$ for $\HGC_{A_m,n}$. For example, the (homotopy) Lie bracket has the following form:
\beq{equ:bracketpic}
\left[ 
\begin{tikzpicture}[baseline=-.8ex]
\node[draw,circle] (v) at (0,.3) {$\Gamma_1$};
\draw (v) edge +(-.5,-.7) edge +(-.25,-.7) edge +(0,-.7) edge +(.25,-.7) edge +(.5,-.7);
\end{tikzpicture}
,
\begin{tikzpicture}[baseline=-.65ex]
\node[draw,circle] (v) at (0,.3) {$\Gamma_2$};
\draw (v) edge +(-.5,-.7) edge +(-.25,-.7) edge +(0,-.7) edge +(.25,-.7) edge +(.5,-.7);
\end{tikzpicture}
\right]
=
\sum
\begin{tikzpicture}[baseline=-.8ex]
\node[draw,circle] (v) at (0,.3) {$\Gamma_1$};
\node[int] (i) at (.7,-.5) {};
\draw (v) edge +(-.5,-.7) edge +(0,-.7) edge (i) edge[bend left] (i) edge[bend right] (i);
\node[draw,circle] (vv) at (1.4,.3) {$\Gamma_2$};
\draw (vv) edge (i) edge[bend left] (i) edge[bend right] (i) edge +(0,-.7) edge +(.7,-.7) (i) edge (.7,-1);
\end{tikzpicture}\, ,
\eeq
where the decorations $\omega$ and~$1$ on hairs are multiplied whenever hairs are joined. The sum is taken over pairs of non-empty 
subsets of hairs of $\Gamma_1$ and $\Gamma_2$. More generally, $\ell_r$ is the sum over $r$-tuples of non-empty
subsets of hairs of $\Gamma_1,\ldots,\Gamma_r$ with every summand being a new connected hairy graph, where all selected hairs are joined into one. With our grading conventions each operation $\ell_r$ (as well as the differential) has degree $-1$. The orientation
of each graph in the sum is obtained by concatenating the orientation sets of $\Gamma_1$, $\ldots$, $\Gamma_r$,
and placing the new vertex and new edge in front. The new (hair) edge is again oriented upward --  towards the new vertex.

\subsection{Connected components and Maurer-Cartan elements}\label{ss:MC}
As it is explained in the introduction, and also stated in Theorem~\ref{thm:emb_long}, one has that
\beq{eq:pi0}
\pi_0\Embbar(S^m,\R^n)=\pi_0\Embbar_\p(\R^m,\R^n), \quad n-m\geq 3,
\eeq
 are isomorphic as (abelian) groups, and all components of $\Embbar(S^m,\R^n)$ (as well as all components of $\Embbar_\p(\R^m,\R^n)$)
have the same homotopy type. These groups are almost always finite except two cases:
\begin{enumerate}[label=(\alph*)]
\item $m=2k-1$, $n=4k-1$, $k\geq 2$;
\item $m=4k-1$, $n=6k$,  $k\geq 1$.
\end{enumerate}
In these two cases this group is infinite of rank one~\cite[Corollary~20]{FTW}.\footnote{This fact can also be easily obtained from 
Haefliger's \cite[Corollary~6.7 and Remark~6.8]{Haefliger2}.} In case~(a), an infinite order generator appears as image, under inclusion
\[
\Omega^{2k}V_{2k-1}(\R^{4k-1})\to \Embbar_\p(\R^{2k-1},\R^{4k-1}),
\]
of the $\SO(2k)$ Euler class in $\pi_{2k}V_{2k-1}(\R^{4k-1})=\pi_{2k}(\SO(4k-1)/\SO(2k))$. In case~(b), an infinite order generator corresponds
to the Haefliger trefoil $S^{4k-1}\hookrightarrow \R^{6k}$ \cite{Haefliger1,Haefliger2}. 

By \cite[Corollary 1.3]{FTW2}, the $L_\infty$-algebras $\HGC_{\bar A_m,n}$ and $\HGC_{A_m,n}$ do provide some information about the sets~\eqref{eq:pi0}.
Namely, one has naturally defined finite-to-one maps
\begin{align*}
m\colon \pi_0\Embbar_\p(\R^m,\R^n)&\to\left.\MC\left(\HGC_{\bar A_m,n}\right)\right/{\sim},\\
m\colon \pi_0\Embbar(S^m,\R^n)&\to\left.\MC\left(\HGC_{A_m,n}\right)\right/{\sim}
\end{align*}
from the sets of connected components to the sets of Maurer-Cartan elements modulo gauge equivalence.
Here, \lq\lq{}finite\rq\rq{} can also mean zero, i.e.
some components are not hit.
 Since the $L_\infty$-algebra $\HGC_{\bar A_m,n}$ is abelian,
\[
\left.\MC\left(\HGC_{\bar A_m,n}\right)\right/{\sim}\, = \, H_0\left(\HGC_{\bar A_m,n}\right).
\]
It is not hard to see that $\HGC_{A_m,n}$ in degrees $\leq 0$ can have  only trees with all hairs labeled by $\omega$  \cite[Proposition~5.1]{FTW2}.
Thus, $H_0(\HGC_{A_m,n})= H_0(\HGC_{\bar A_m,n})$ and $\left.\MC(\HGC_{A_m,n})\right/{\sim}\, =\, \left.\MC(\HGC_{\bar A_m,n})\right/{\sim}$,
see \cite[Corollary~ 5.2]{FTW2}. By \cite[Remark~19]{FTW}, 
\[
\Q\otimes\pi_0\Embbar_\p(\R^m,\R^n) \simeq H_0\left(\HGC_{\bar A_m,n}\right).
\]
The latter group is non-trivial (and is $\Q$) exactly in two cases (a) and (b) above. The  Maurer-Cartan elements corresponding to case~(a) are 
multiples of the line graph
\[
L_\omega=
\begin{tikzpicture}[baseline=-.65ex]
\node (v) at (0,0) {$\omega$};
\node (w) at (1,0) {$\omega$};
\draw (v) edge (w);
\end{tikzpicture}.
\]
For case (b), such elements are multiples of the tripod
\[
T_\omega=
\begin{tikzpicture}[baseline=0]
\node (v1) at (-.6,0) {$\omega$};
\node (v2) at (0,0) {$\omega$};
\node (v3) at (.6,0) {$\omega$};
\node [int] (w) at (0,.6) {};
\draw (w) edge (v1) edge (v2) edge (v3);
\end{tikzpicture}.
\]
By \cite[Corollary 1.3]{FTW2}, for an embedding $\psi\in \Embbar_\p(\R^m,\R^n)$ (respectively, $\psi\in \Embbar(S^m,\R^n)$), the rational homotopy
type of the component $\Embbar_\p(\R^m,\R^n)_\psi$ (respectively, $\Embbar(S^m,\R^n)_\psi$) is expressed by the positive degree truncation
of the  $m(\psi)$ twisted $L_\infty$-algebra $\bigl(\HGC_{\bar A_m,n}^{m(\psi)}\bigr){}_{>0}$ (respectively,  $\bigl(\HGC_{A_m,n}^{m(\psi)}\bigr){}_{>0}$).\footnote{See \cite[section 2.1]{DolRog} for the definition of the twisted $L_\infty$-structure on an $L_\infty$-algebra.} Since the
$L_\infty$-algebra $\HGC_{\bar A_m,n}$ is abelian, such a twist has no effect on it, which corresponds to the fact that all connected components of a loop space
have the same homotopy type. In case~(a), the twist by $L_\omega$ changes neither the differential nor the bracket of $\HGC_{A_{2k-1},4k-1}$. 
This is because
for even codimension $n-m$, any graph with  two $\omega$-hairs attached to an internal vertex is zero. The twisting by $T_\omega$ does affect the differential
and the $L_\infty$ structure of $\HGC_{A_{4k-1},6k}$. We do not do it here, but one can show that $\HGC_{A_{4k-1},6k}^{T_\omega}$
is $L_\infty$ isomorphic to the non-deformed one $\HGC_{A_{4k-1},6k}$, which confirms the fact that all components of $\Embbar(S^m,\R^n)$, $n-m\geq 3$,
have the same homotopy type.

\subsection{Computations}
The relation between the long and non-long embedding spaces can be reproduced combinatorially on graph-complexes as follows.
There is an obvious inclusion of $L_\infty$-algebras $\HGC_{\bar A_m,n}\to \HGC_{A_m,n}$ corresponding to the inclusion $\bar A_m\to A_m$.

Let us also consider the following low degree diagrams which are non-zero for certain values of $m$ and $n$.
\begin{align*}
L&=
\begin{tikzpicture}[baseline=-.65ex]
\node (v) at (0,0) {$1$};
\node (w) at (1,0) {$\omega$};
\draw (v) edge (w);
\end{tikzpicture}
&
D&=
\begin{tikzpicture}[baseline=-.65ex,every loop/.style={}]
\node (v) at (0,0) {$\omega$};
\node [int] (w) at (1,0) {};
\draw (v) edge (w) (w) edge[loop] (w);
\end{tikzpicture}
&
T&=
\begin{tikzpicture}[baseline=0]
\node (v1) at (-.6,0) {$1$};
\node (v2) at (0,0) {$\omega$};
\node (v3) at (.6,0) {$\omega$};
\node [int] (w) at (0,.6) {};
\draw (w) edge (v1) edge (v2) edge (v3);
\end{tikzpicture}
\end{align*}
The graph $L$ is of degree $n-m-1$ and is always non-zero; $D$ is of degree $n-m-2$ and is nonzero if $n$ is even; and $T$ is of degree $2n-2m-3$ and is nonzero if and only if $n-m$ is odd.
One has that $dL=D$, i.e., for even $n$ the corresponding classes cancel in homology in $\HGC_{A_m,n}$.
(But not in $\HGC_{\bar A_m,n}$, since $L\notin \HGC_{\bar A_m,n}$.)
Also note that $dT=0$.
Using these classes we can completely describe the relation between $\HGC_{A_m,n}$ and $\HGC_{\bar A_m,n}$ as follows.

\begin{thm}\label{prop:HGCAAbar}
The mapping cone $C$ of the inclusion $\HGC_{\bar A_m,n}\to\HGC_{A_m,n}$ has the following homology, depending on $m$ and $n$:
\begin{itemize}
\item For $m,n$ even $H(C)$ is one-dimensional, spanned by a class whose projection to $\HGC_{\bar A_m,n}$ is $D$.
\item For $n$ even and $m$ odd $H(C)$ is two-dimensional, spanned by a class corresponding to $D$ in $\HGC_{\bar A_m,n}$ as before and the class of $T\in\HGC_{A_m,n}$.
\item For $m,n$ odd $H(C)$ is one-dimensional, spanned by the class of $L$ in $\HGC_{A_m,n}$.
\item For $n$ odd and $m$ even $H(C)$ is two-dimensional, spanned by the class of $L$ and $T$ in $\HGC_{A_m,n}$.
\end{itemize}
\end{thm}

\begin{rem}\label{rem:Qranks}
Theorem~\ref{prop:HGCAAbar} provides a different proof of Corollary~\ref{cor:Qranks}.
\end{rem}

The result can alternatively be reformulated as follows.
\begin{cor}\label{cor:HGCAAbar1}
Let $U^t\subset \HGC_{A_m,n}$ be the subspace spanned by trees with exactly one $1$-decorated hair. 
Consider the vector space direct sum $U^t\oplus \HGC_{\bar A_m,n}\subset \HGC_{A_m,n}$ with the induced (subspace) $L_\infty$-structure.
Then the inclusion $U^t\oplus \HGC_{\bar A_m,n}\to \HGC_{A_m,n}$ is a quasi-isomorphism of $L_\infty$-algebras.
\end{cor}

As an immediate consequence,
the $L_\infty$-algebra $\HGC_{A_m,n}$ is homotopy abelian for $n-m$ even.
Indeed,  for $n-m$ even there can be at most one $\omega$-hair attached to a vertex by symmetry.
 (In particular,
this means that $U^t$ is one-dimensional and  is spanned by~$L$.)
But then the statement easily follows from Corollary~\ref{cor:HGCAAbar1}, since all possible higher $L_\infty$-operations necessarily produce multiple $\omega$-hairs at some vertex.
Less trivially, the above arguments can also be extended to show that $\HGC_{A_m,n}$ is homotopy abelian for $n$ even and $m$ odd. This 
gives a different proof of the first statement of Corollary~\ref{cor:rht}. 
Similarly, we can also recover the second statement of Corollary~\ref{cor:rht}, which is immediate in case both $m$  and $n$ are odd.
In the remaining case $n$ odd and $m$ even, there is a nontrivial bracket, namely
\[
[L,L] = T,
\]
so that $\HGC_{A_m,n}$ is not homotopy abelian. It is possible to upgrade the map $\Phi$ that we construct below, see Lemma~\ref{lem:Phi}, to an $L_\infty$-map, see the footnote at the end of the proof of Lemma~\ref{lem:Phi}, 
that would allow one to split off $L$ and $T$ - the two classes coming from $S^{n-m-1}$, as an $L_\infty$-direct summand.  

To prepare for the proof of Theorem~\ref{prop:HGCAAbar} let us introduce the non-unital dgca
\[
A'_m = \Q\epsilon \oplus \Q\omega
\]
with $\epsilon$ of degree $0$ and $\omega$ of degree $m$, and products $\epsilon^2=\epsilon$ and $\epsilon\omega=\omega^2=0$.
We consider the hairy graph-complex $\HGC_{A'_m,n}$. Note also that the complexes $\HGC_{A'_m,n}$ and $\HGC_{A_m,n}$ are isomorphic as graded vector spaces, identifying $\epsilon$ and $1$. In fact, from now on we shall tacitly identify the decorations $\epsilon$ and $1$ on hairs of graphs, keeping in mind however that the differentials on $\HGC_{A'_m,n}$ and $\HGC_{A_m,n}$ are different.
Concretely, the differential in $\HGC_{A_m,n}$ has pieces fusing several $1$-decorated hairs with one $\omega$-decorated hair, and these terms are absent in the differential on $\HGC_{A'_m,n}$.
Note that there is again an inclusion $\HGC_{\bar A_m,n}\to \HGC_{A'_m,n}$.
\begin{lemma}\label{lem:HGCAp}
The inclusion map $\Q L \oplus \Q T \oplus \HGC_{\bar A_m,n}\to \HGC_{A'_m,n}$ is a quasi-isomorphism.
Here we understand that $\Q T:=0$ in case $n-m$ is even since then $T=0$.
\end{lemma}

\begin{rem}\label{rem:ABC}
It follows from~\cite[Corollary~1.3]{FTW2} that the complex $\HGC_{A'_m,n}$ computes the rational homotopy groups  of the space 
$\Embbar_\p(\R^m\sqcup\{*\},\R^n)$. On the other hand, by Lemma~\ref{l:ABC},
 $\Embbar_\p(\R^m\sqcup\{*\},\R^n)\simeq \Embbar_\p(\R^m,\R^n)\times S^{n-m-1}$. This explains the quasi-isomorphism of Lemma~\ref{lem:HGCAp}.
\end{rem}

\begin{proof}[Proof of Lemma~\ref{lem:HGCAp}.]
There is a splitting of complexes 
\[
\HGC_{A'_m,n} = \HGC_{\bar A_m,n} \oplus U
\]
with $U$ being the subcomplex spanned by graphs with at least one hair labeled $\epsilon$.
Our goal is to show that $H(U)$ is one- or two-dimensional.
To do this we may follow the proof of \cite[Theorem 1]{KWZ2}.
First note that the differential creates exactly one vertex, hence the homology of $U$ is graded by the number of vertices.
Let $U^t\subset U$ be the subcomplex spanned by trees with exactly one hair labeled $\epsilon$.
It is an easy exercise to check that $H(U^t)=\Q L\oplus \Q T$.
We are going to show by induction on the number of internal vertices that the inclusion $U^t\subset U$ is a quasi-isomorphism.
For zero internal vertices the statement is quickly checked by hand. Suppose we know the statement for less than $k$ internal vertices, and we desire to prove it for $k$ internal vertices.
Consider the splitting
\[
U=U_1\oplus U_{>1} 
\]
where $U_1$ is spanned by diagrams having exactly one $\epsilon$-labeled hair, and $U_{>1}$ being spanned by diagrams having at least two such hairs. The space $U_1$ is
preserved by the differential.
One may set up a bounded spectral sequence such that the lowest page differential is the component $f:U_{>1}\to U_1$ that creates one new internal vertex with an $\epsilon$-hair, connecting all $\epsilon$-hairs to it. Indeed, the complex 
$\HGC_{A'_m,n}$ as well as $U$ is a direct sum of finite complexes as the differential preserves the number of edges
minus the number of internal vertices. This number is sometimes called {\it complexity}
and the number of graphs in $\HGC_{A'_m,n}$ of any given complexity is finite. To obtain the spectral sequence
in question we  filter $U$ (or rather each its complexity summand) by  $\rho$ minus the number of vertices, where
we set  $\rho(U_1)=1$ and $\rho(U_{>0})=0$. The differential $d_0$ of the induced spectral sequence is exactly the map $f$.

The map $f$ is injective. The cokernel $V:=\mathit{coker} f$ consists of $L$ and graphs which become disconnected upon removing the  vertex at the $\epsilon$-hair.
Going further, one may filter $V$ by the number of connected components at that vertex.
On the associated graded the complex obtained is just a symmetric power of the complex $U$.
Hence, invoking the induction hypothesis, we have shown the desired statement.
\end{proof}

Our next goal is to compare the complexes $\HGC_{A'_m,n}$ and $\HGC_{A_m,n}$. Note that the algebra
$A'_m$ can not be obtained as an associated graded of $A_m$. Neither $\HGC_{A'_m,n}$ is an associated graded of
$\HGC_{A_m,n}$. So, we need a more subtle argument to compare their homology. Our strategy will be to split each of the two complexes into three pieces: an acyclic one, a small piece
where the two complexes differ, and the {\it main part} that we show to be the same up to a non-trivial isomorphism.
To this end we will consider the subcomplexes $\HGC_{A'_m,n}'\subset \HGC_{A'_m,n}$ and $\HGC_{A_m,n}'\subset \HGC_{A_m,n}$ spanned by all the diagrams with at least one $\omega$-labeled hair excluding $L$ and $D$. These subcomplexes are our {\it main parts}.

Consider now a hairy graph $\Gamma\in\HGC_{A'_m,n}'$ and let $S$ be some subset of the hairs decorated by $\epsilon$ in $\Gamma$.
Denote by $R_S(\Gamma)$ the sum of all graphs obtained by reconnecting the hairs in $S$ to internal vertices of $\Gamma$, not forming tadpoles\footnote{I.e., a hair cannot be connected to the internal vertex it attaches to.}, pictorially
\[
\begin{tikzpicture}[baseline=-.65ex,decoration={brace,mirror,amplitude=3}]
\node[draw, circle] (v) at (0,.3) {$\Gamma$};
\draw (v) edge +(-.5,-.6) edge +(-.25,-.6)  edge +(0,-.6) edge +(.25,-.6) edge +(.5,-.6);  
\draw [decorate] (-.5,-.5) -- (-.25,-.5) 
node [pos=0.5,anchor=north,yshift=-0.1cm] {$S$}; 
\end{tikzpicture}
\quad
\mapsto
\quad
R_S(\Gamma)=
\sum
\begin{tikzpicture}[baseline=-.65ex]
\node[draw, circle] (v) at (0,.3) {$\Gamma$};
\draw (v) edge[out=-135,in=135, loop] (v) edge[out=-115,in=115, loop] (v)  edge +(0,-.6) edge +(.25,-.6) edge +(.5,-.6);  
\end{tikzpicture}\, .
\]
Note that each graph in the sum has the same orientation set (of vertices, edges, and $\omega$-labels). So, we keep the same order 
of their orientation sets and the same orientation of edges. With this convention, no signs appear in the sum above.

Now consider the map 
\[
\Phi : \HGC_{A'_m,n}' \to \HGC_{A_m,n}'
\]
which is defined combinatorially by the formula
\[
\Phi(\Gamma) = 
(-1)^{\#\epsilon}
\sum_S R_S(\Gamma).
\]
where $\Gamma\in \HGC_{A'_m,n}'$ is a graph with $\#\epsilon$ many $\epsilon$-decorated hairs.

\begin{lemma}\label{lem:Phi}
The map $\Phi:\HGC_{A'_m,n}' \to \HGC_{A_m,n}'$ is an isomorphism of complexes.
\end{lemma}
\begin{proof}
It is clear that the map is an isomorphism since $\Phi(\Gamma)=\pm \Gamma+(\cdots)$, with $(\cdots)$ representing terms of loop orders higher than that of $\Gamma$.
We next show that $\Phi$ commutes with the differentials. The only graph in $\HGC_{A'_m,n}' $
that does not have internal vertices is $
L_\omega=
\begin{tikzpicture}[baseline=-.65ex]
\node (v) at (0,0) {$\omega$};
\node (w) at (1,0) {$\omega$};
\draw (v) edge (w);
\end{tikzpicture}$. Since $d'(L_\omega)=d(L_\omega)=0$ and $\Phi(L_\omega)=L_\omega$,
this graph can be ignored and from now on we only consider graphs that have internal vertices.
Let us first reformulate the problem. We identify $\HGC_{A'_m,n}'$ and $\HGC_{A_m,n}'$ as graded vector spaces, and denote the differential of $\HGC_{A'_m,n}'$ by $d'$ and that of $\HGC_{A_m,n}'$ by $d$.
Let $s: \HGC_{A'_m,n}' \to \HGC_{A'_m,n}'$ be the map of graded vector spaces that reconnects one hair $h$ labeled $\epsilon$ to an internal vertex (but not the one from which $h$ is growing).
\[
s(\Gamma) = 
\sum
\begin{tikzpicture}[baseline=-.65ex]
\node[draw, circle] (v) at (0,.3) {$\Gamma$};
\draw (v) edge[out=-135,in=135, loop] (v)  edge +(0,-.6) edge +(.25,-.6) edge +(.5,-.6);  
\end{tikzpicture}
\]
Then we can write $\Phi=\exp(s)\circ I_\epsilon$, where $I_\epsilon(\Gamma)=(-1)^{\#\epsilon}\Gamma$.
We desire to show that 
$\Phi\circ d' = d \circ \Phi$, or equivalently
\[
\exp(\ad_s) (\underbrace{I_\epsilon d' I_\epsilon}_{=:\bar d'}) = \sum_{j=0}^\infty 
\frac{1}{j!}\ad_s^j \bar d' \stackrel{?}{=} d,
\]
where $\bar d'=I_\epsilon d' I_\epsilon$ and $\ad_s=[s,-]$ is the commutator as usual.

Furthermore, let us split the differential $d'$, and similarly $\bar d'$ in several pieces.
To this end it is most convenient to temporarily enlarge our complex $\HGC_{A'_m,n}'$ in that we also allow graphs with univalent and bivalent internal vertices.
Then we split
\[
d' = d'_1 - B_\emptyset + d'_\epsilon + d'_\omega
\]
into the following four terms.
\begin{itemize}
	\item $d'_1$ splits a vertex into two vertices, distributing the incoming edges in all possible ways, including such that create uni- or bivalent internal vertices.
	\item $B_\emptyset$ attaches a new univalent vertex to the graph. The sign is such that it precisely cancels those terms from $d'_1$ that create univalent internal vertices.
	
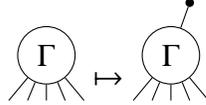
\[
	\begin{tikzpicture}[baseline=-.65ex,decoration={brace,mirror,amplitude=3}]
	\node[draw, circle] (v) at (0,.3) {$\Gamma$};
	\draw (v) edge +(-.5,-.6) edge +(-.25,-.6)  edge +(0,-.6) edge +(.25,-.6) edge +(.5,-.6);  
	\end{tikzpicture}
	\mapsto
	\begin{tikzpicture}[baseline=-.65ex,decoration={brace,mirror,amplitude=3}]
	\node[draw, circle] (v) at (0,.3) {$\Gamma$};
	\draw (v) edge +(-.5,-.6) edge +(-.25,-.6)  edge +(0,-.6) edge +(.25,-.6) edge +(.5,-.6);  
	\node[int] (w) at (0.25,1) {};
	\draw (v) edge (w);
	\end{tikzpicture}
	\] 
	\item $d'_\epsilon$ creates a new internal vertex with an $\epsilon$-decorated hair and attaches a non-empty subset of the $\epsilon$-decorated hairs to it,
	\[
	d'_\epsilon \Gamma = \sum_{K \atop |K|\geq 1} A_K(\Gamma),
	\]
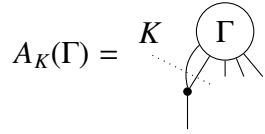
	\[
	A_K(\Gamma) = 
	\begin{tikzpicture}[baseline=-.65ex]
\node[draw, circle] (v) at (0,.3) {$\Gamma$};
\node[int] (w) at (-0.5,-0.5) {};
\draw (v) edge (w) edge[bend right] (w) 
(v)  edge +(0,-.6) edge +(.25,-.6) edge +(.5,-.6) (w) edge +(0,-.5);  
\draw[dotted] (-0.2,-0.4) -- (-1,0) node[above] {$K$};
\end{tikzpicture}
	\]
	\item $d'_\omega$ creates a bivalent internal vertex on an  $\omega$-decorated hair,
	\[
	d'_\omega =  C_\emptyset(\Gamma).
	\]
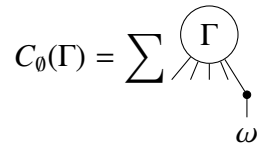
	\[
	C_\emptyset(\Gamma) = \sum
\begin{tikzpicture}[baseline=-.65ex,decoration={brace,mirror,amplitude=3}]
\node[draw, circle] (v) at (0,.3) {$\Gamma$};
\node[int] (w) at (.5,-.5) {};
\draw (v) edge +(-.5,-.6) edge +(-.25,-.6)  edge +(0,-.6) edge +(.25,-.6) edge (w) (w) ++(0,-.3) node[below] {$\omega$} edge (w);  
\end{tikzpicture}
	\]
	Each operation above produces a sum of graphs $\Gamma'$ that have one more vertex and one more edge
	than the graph~$\Gamma$. So, we put these two new elements as the first and second elements
	of the orientation set of $\Gamma'$ keeping without change the rest. The new edge in $\Gamma'$ is always
	oriented towards the new vertex. In the case of $d'_\omega$ (as well as in the case of $d'_\epsilon$) the new edge is considered to be 
	the hair one. 
	
	Note that the $|K|=1$-term of $d'_\epsilon$ and $d'_\omega$ together cancel all terms in the total differential $d'$ that possibly create a graph with a bivalent internal vertex.

\item Finally we note that $I_\epsilon d'_1I_\epsilon = d_1'$, $I_\epsilon B_\emptyset I_\epsilon=B_\emptyset$,  and $I_\epsilon d'_\omega I_\epsilon = d_1'$. Denoting 
$\bar d'_\omega := (I_\epsilon d'_\omega I_\epsilon)$
we furthermore have 
\begin{align*}
\bar d'_\epsilon\Gamma &= \sum_{K \atop |K|\geq 1} (-1)^{|K|-1} A_K(\Gamma) .
\end{align*}

\end{itemize}
One quickly checks that $[s,d_1']=0$. Note that in $s(d_1'(\Gamma))$ the part that comes from connecting an $\epsilon$-hair $h$ to a new vertex created by $d_1'$ by blowing up the vertex to which $h$ is attached, is zero. Indeed, when
$n$ is even each such graph  is zero as it contains a double edge. When $n$ is odd,  the sum can be seen as a sum of 
pairs of  identical graphs  
with an edge, former $h$,  appearing with the opposite orientation. Thus,  two such graphs cancel each other. 
Furthermore,
\beq{eq:B_J}
\frac 1 {j!} (\ad_s^j B_\emptyset)(\Gamma) = \sum_{J\atop |J|=j} B_J(\Gamma), 
\eeq
where the sum is over subsets $J$ of the set of $\epsilon$-labeled hairs and
\[
B_J(\Gamma) =\sum
\begin{tikzpicture}[baseline=-.65ex]
\node[draw, circle] (v) at (0,.3) {$\Gamma$};
\node[int] (w) at (-0.8,0.4) {};
\draw (w) edge (v);
\draw (v) edge[out=-135,in=-35] (w) edge[out=-115,in=-90] (w)  edge +(0,-.6) edge +(.25,-.6) edge +(.5,-.6);  
\draw[dotted] (-0.5,0.2) -- (-1.1,0) node[below] {$J$};
\end{tikzpicture}
\]
is obtained by connecting the hairs $J$ to a new vertex and that furthermore to an arbitrary existing vertex of $\Gamma$. Indeed, if we denote by $B_j(\Gamma)$ the right-hand side of~\eqref{eq:B_J},
one has
\beq{eq:B_j}
s(B_j(\Gamma))=B_j(s(\Gamma))+(j+1)B_{j+1}(\Gamma),
\eeq
i.e., $[s,B_j]=(j+1)B_{j+1}$,
which applying induction proves~\eqref{eq:B_J}.
Next, 
\[
\frac 1 {j!} (\ad_s^j \bar d'_\epsilon)(\Gamma) =
\sum_{\substack{ J,K \\ J\cap K=\emptyset \\ |J|=j,|K|\geq 1 }}
(-1)^{|K|-1} A_{J\cup K}(\Gamma)
+
\sum_{\substack{J,K \\ J\cap K=\emptyset \\ |J|=j-1,|K|\geq 1}}
(-1)^{|K|-1} B_{J\cup K}(\Gamma).
\]
To prove it, denote by $A_{j;>0}(\Gamma)$ the first sum and by $B_{j-1;>0}(\Gamma)$ the second one. The above is proved  by checking that the operations $A_{j;>0}$ and $B_{j-1;>0}$ satisfy
\[
[s,A_{j;>0}]  = (j+1)A_{j+1;>0} + B_{j;>0}\,\,\, ,\hspace{15pt}
\hspace{15pt}
[s,B_{j-1;>0}]= j B_{j;>0}.
\]
Finally, 
\[
\frac 1 {j!} (\ad_s^j  d'_\omega)(\Gamma) =
\sum_{\substack{ J \\ |J|=j  }}
 C_{J}(\Gamma)\ ,
\]
where the sum is again over subsets $J$ of the $\epsilon$-decorated hairs, and $C_J(\Gamma)$ is obtained by connecting the hairs in $J$ to one new vertex attached to an $\omega$-decorated hair. 
\[
C_J(\Gamma)=\sum
\begin{tikzpicture}[baseline=-.65ex,decoration={brace,mirror,amplitude=3}]
\node[draw, circle] (v) at (0,.3) {$\Gamma$};
\node[int] (w) at (.5,-.5) {};
\draw (v) edge +(-.5,-.6) 
edge[out=-80] (w) edge[out=-110] (w) edge (w) 
(w) ++(0,-.3) node[below] {$\omega$} edge (w);  
\draw[dotted] (0.2,-0.2) -- (-0.2,-0.5) node[below] {$J$};
\end{tikzpicture}
\]
Now, putting everything together we get (with the sums being over subsets of the $\epsilon$-decorated hairs)
\begin{multline*}
(\Phi d'\Phi^{-1})(\Gamma) = \sum_{j=0}^\infty\frac 1{j!} (\ad_s^j \bar d') (\Gamma)=\\
d_1'(\Gamma)
-\sum_{J\atop |J|\geq 0} B_J(\Gamma)
+
\sum_{\substack{ J,K \\ J\cap K=\emptyset \\ |J|\geq 0,|K|\geq 1 }}
(-1)^{|K|-1} A_{J\cup K}(\Gamma)
+
\sum_{\substack{J,K \\ J\cap K=\emptyset \\ |J|\geq 0,|K|\geq 1}}
(-1)^{|K|-1} B_{J\cup K}(\Gamma)
+
\sum_{\substack{ J \\ |J|\geq 0  }}
C_{J}(\Gamma) 
\\
=d_1' (\Gamma)
-
\sum_{\substack{J,K \\ J\cap K=\emptyset \\ |J|,|K|\geq 0}}
(-1)^{|K|} B_{J\cup K}(\Gamma)
-
\sum_{\substack{ J,K \\ J\cap K=\emptyset \\ |J|\geq 0,|K|\geq 1 }}
(-1)^{|K|} A_{J\cup K}(\Gamma)
+
\sum_{\substack{ J \\ |J|\geq 0  }}
C_{J}(\Gamma) \, .
\end{multline*}
Now we use (twice) that for any function $J\mapsto X_J$ on subsets as above
\[
\sum_{J\cap K=\emptyset} (-1)^{|K|} X_{J\cup K} = X_\emptyset. 
\]
This simplifies the above expression to 
\[
d_1' (\Gamma)
-
B_\emptyset(\Gamma)
+
\sum_{\substack{ J \\ |J|\geq 1 }}
A_{J}(\Gamma)
+
\sum_{\substack{ J \\ |J|\geq 0  }}
C_{J}(\Gamma) 
= d(\Gamma).
\]
This is precisely $d$, hence the Lemma is proven.\footnote{The same construction applied to disconnected graphs, interpreted as the Chevalley complex, in fact can be used to construct an $L_\infty$-isomorphism, not just one of complexes.}
\end{proof}
Let us finish the proof of Theorem~\ref{prop:HGCAAbar}.
\begin{proof}[Proof of Theorem~\ref{prop:HGCAAbar}.]

Now let $\HGC_{\bar A_m,n}'\subset\HGC_{\bar A_m,n}$ be the subcomplex spanned by all the diagrams excluding $D$. 
Note that we have a natural inclusion $\HGC_{\bar A_m,n}'\to \HGC_{A'_m,n}'$, fitting into the commutative diagram
\[
\begin{tikzcd}
\HGC_{\bar A_m,n}' \ar{r} \ar{dr} &  \HGC_{A'_m,n}' \ar{d}{\Phi}[swap]{\cong} \\
 & \HGC_{A_m,n}' 
\end{tikzcd}\, .
\]
From Lemma \ref{lem:HGCAp} and its proof we see that
\[
H(\HGC_{A_m,n}') \cong H(\HGC_{A'_m,n}')\cong H(\HGC_{\bar A_m,n}')\oplus \Q T,
\]
again using the convention that $\Q T=0$ if $T=0$.

To show Theorem~\ref{prop:HGCAAbar} it just remains to compare the homology of $\HGC_{A_m,n}'$ and $\HGC_{A_m,n}$. The complex $\HGC_{A_m,n}$ is a direct sum of three complexes
\beq{eq:3parts}
\HGC_{A_m,n} = W_0 \oplus \HGC_{A_m,n}'\oplus \bigl(\Q L\oplus \Q D),
\eeq
with $W_0$ spanned by graphs with zero $\omega$-vertices.
It is shown in \cite[Theorem~1]{KWZ2} (see also Lemma \ref{lem:HGCAp}) that $H(W_0)=0$.
Now the last summand in~\eqref{eq:3parts} has a non-trivial differential (sending $L$ to $D$) iff $n$ is even, i.e. iff $D\neq 0$.
 Combining the above observations, depending on the parity of $n$ and $n-m$ we arrive at Theorem~\ref{prop:HGCAAbar}.
\end{proof}

\section{Codimension $n-m\leq 2$}\label{s4}

\subsection{Codimension one}\label{ss:41}

\begin{prop}\label{p:codim1}
For $n\geq 2$, one has an equivalence
\[
\Embbar(S^{n-1},\R^n)\simeq
\begin{cases}
S^0\times\Embbar_\partial(\R^{n-1},\R^n),& n=3 \text{ or }7;\\
\Embbar_\partial(\R^{n-1},\R^n),&\text{  otherwise.}
\end{cases}
\]
\end{prop}

\begin{proof}
It is easy to see that the statement of Proposition~\ref{prop:hofib1} holds for $n-m=1$. The crucial fact is that the complement of any long knot $\R^{n-1}\hookrightarrow\R^n$ is homotopy equivalent 
to $S^0$, which follows from the generalized Schoenflies theorem~\cite{Mazur}. Any codimension one long knot is
regularly homotopic to the trivial one \cite[Theorem~2]{Kaiser}, which means that the induced map
\[
D_*\colon\pi_0\Emb_\partial(\R^{n-1},\R^n)\to \pi_0\Omega^{n-1}V_{n-1}(\R^n)=\pi_{n-1}\SO(n)
\]
is zero. On the other hand, the map
\[
h_*\colon\pi_0S^0\to\pi_{n-1}\SO(n)
\]
is trivial iff $S^{n-1}$ can be eversed in $\R^{n}$, i.e., iff $n=3$ or~7 \cite{Smale,Kaiser}. The result follows.
\end{proof}

It is known that $\pi_0\Emb_\partial(\R^{n-1},\R^n)=\theta_n$, $n\neq 4$, -- the group of $n$-spheres,
see~\cite[Section~5]{Budney2} and references within. (For $n=4$, the question whether the space is connected is
equivalent to the smooth Schoenflies problem -- does a smoothly embedded $S^3$ in $\R^4$ always bound the standard $D^4$, which is still open in this dimension.) By the same argument as in the proof of Theorem~\ref{thm:emb_long},
one gets
\[
\Embbar(S^{n-1},\R^n)\simeq\hofiber\left(S^0\to B\Embbar_\partial(\R^{n-1},\R^n)\times \pi_{n-1}\SO(n)\right),\,\, n\neq 4,
\]
since $\Omega^{n-1}\SO(n)\sslash\Emb_\partial(\R^{n-1},\R^n)\simeq  B\Embbar_\partial(\R^{n-1},\R^n)
\times \pi_{n-1}\SO(n)$. We conclude that the main statement of Theorem~\ref{thm:emb_long} 
holds for $n=m+1$ iff $n=3$ or~7.

\subsection{Codimension two}\label{ss42}
The main statement of Theorem~\ref{thm:emb_long}  always fails in codimension $n-m=2$. There are two reasons for this.
Firstly, for $n\geq 3$, neither $\Emb_\partial(\R^{n-2},\R^n)$, nor $\Embbar_\partial(\R^{n-2},\R^n)$ are loop spaces
\cite[Proposition~5.11]{Budney2}\footnote{It was pointed out to us by R.~Budney that
the proof of this proposition has a little mistake that can easily be corrected. Contrary to what is said, there are codimension two long knots $f$ with exterior
$C_f\nsimeq S^1$ and $\pi_1 C_f=\Z$. Such knots are studied in~\cite{HilKea}. However, only 
the trivial non-parametrized knot $S^{n-2}\hookrightarrow S^n$ is invertible according to~\cite{Sos,Maeda1,Maeda2}.}.
Secondly, the complement $C_f$ of a long knot $f\colon\R^{n-2}\hookrightarrow\R^n$ almost always 
is not weakly equivalent to $S^1$ \cite{Sos}. Thus, Proposition~\ref{prop:hofib1} does not hold for $n=m+2$. 
Indeed, the space $\Emb_\partial(\R^{n-2},\R^n\setminus\{0\})$ is weakly homotopy
equivalent to $\Emb_\partial(\R^{n-2}\sqcup\{*\},\R^n)$ (by the same argument as in Lemma~\ref{l:ABC}),
but the latter space is a possibly non-trivial fiber bundle over $\Emb_\partial(\R^{n-2},\R^n)$ with fiber~$C_f$.

\subsection{Goodwillie-Weiss calculus and graph-complexes}\label{ss:43}
Given a (formally) immersed manifold $M$ in $\R^n$, one can consider the functor $\Embbar(-,\R^n)$
and its objectwise rationalization $\Embbar(-,\R^n)^\Q$ on the poset of open sets of $M$.
Goodwillie-Weiss calculus~\cite{GW,Weiss} produces Taylor towers of approximations to
these two functors:
\beq{eq:Taylor}
\Embbar(M,\R^n) \to T_\infty \Embbar(M,\R^n) \to T_\infty \Embbar(M,\R^n)^\Q.
\eeq
In case codimension is $\geq 3$, the first map is an equivalence \cite{GK,GKW}, and the second
map is finite-to-one on $\pi_0$ and a  rational equivalence on connected components \cite[Section~4.2]{FTW}. 
Even when the codimension condition is not satisfied, it can still be interesting to know what is the right-hand 
side space of~\eqref{eq:Taylor} as it can provide interesting invariants or more generally  cohomology classes
of the embedding space in question. In~\cite[Theorem~1.1]{FTW2}, B.~Fresse and the authors computed
$T_\infty \Embbar(M,\R^n)^\Q$ expressing it as the simplicial set of Maurer-Cartan elements of associated
$L_\infty$-algebra of hairy graph-complexes, provided $M$ is immersible or formally immersible in $\R^{n-2}$. In
particular, one has
\begin{align}
T_\infty \Embbar_\partial(\R^m,\R^n)^\Q&\simeq \MC_\bullet(\HGC_{\bar A_m,n}),
\,\, n-m\geq 2; \label{eq:MC1} \\
T_\infty \Embbar(S^m,\R^n)^\Q&\simeq \MC_\bullet (\HGC_{ A_m,n}),\,\,
 n-m\geq 3 \text{ or } n=m+2=3, 5\text{ or } 9. \label{eq:MC2}
\end{align}
(One needs $S^{n-2}$ to be parallelizable to be formally immersible in $\R^{n-2}$, which is only true for
$S^1$, $S^3$, and $S^7$ \cite{Kaiser,Smale}.)  When the codimension $n-m=2$, the hairy graph-complexes are no more of finite type and their
elements are infinite series of graphs. The graph-complexes in question are considered as completed pronilpotent $L_\infty$-algebras, the completion being taken with
respect to the complexity filtration, see Proof of Lemma~\ref{lem:HGCAp}. 
 Since the $L_\infty$-structure of $\HGC_{\bar A_m,n}$ is abelian, 
each space~\eqref{eq:MC1} is a product of Eilenberg-MacLane spaces
\[
T_\infty \Embbar_\partial(\R^m,\R^n)^\Q\simeq \prod_{i=0}^\infty K\left(H_i(\HGC_{\bar A_m,n}),\,i\,\right), \,\, n-m\geq 2.
\]
In particular this means that for $n-m\geq 2$,
\[
\pi_0 T_\infty \Embbar_\partial(\R^m,\R^n)^\Q=\MC(\HGC_{ \bar A_m,n})/{\sim} = H_0(\HGC_{ \bar A_m,n}).
\]

The statements of Theorem~\ref{prop:HGCAAbar} and Corollary~\ref{cor:HGCAAbar1}  hold for any $m$ and $n$, in particular they
are also true in codimension $n-m=2$. The inclusion $U^t\oplus \HGC_{\bar A_m,n}\subset \HGC_{A_m,n}$ is a quasi-isomorphism 
of filtered (by complexity) completed $L_\infty$-algebras, which induces a quasi-isomorphism of associated graded complexes. By the generalized
 Goldman-Millson theorem~\cite{DolRog}, this inclusion induces an  equivalence  of simplicial sets
 \[
 \MC_\bullet(U^t\oplus \HGC_{\bar A_m,n})\simeq \MC_\bullet\left( \HGC_{A_m,n}\right).
 \]
 As a consequence, in the range of equivalence~\eqref{eq:MC2}, one has
\begin{multline*}
\pi_0 T_\infty \Embbar(S^m,\R^n)^\Q =\MC(\HGC_{ A_m,n})/{\sim} = \MC(U^t\oplus \HGC_{\bar A_m,n})/{\sim}= \\
\MC(\HGC_{ \bar A_m,n})/{\sim} = H_0(\HGC_{ \bar A_m,n})
=H_0(\HGC_{ A_m,n}).
\end{multline*}
Indeed, for $n-m\geq 3$ the third and last equalities are true by degree reasons, see \cite[Corollary~5.2]{FTW2}, while for $n-m=2$ and $n$ odd, $U^t$ is
a direct summand one-dimensional $L_\infty$-subalgebra of $U^t\oplus \HGC_{\bar A_m,n}$ (spanned by $L$ of degree~$1$). Thus, one has
\[
T_\infty \Embbar(S^{n-2},\R^n)^\Q\simeq K(\Q,1)\times T_\infty\Embbar_\partial(\R^{n-2},\R^n)^\Q,\,\,
 n=3,\, 5\text{ or } 9.
 \]

One does not know yet how to express algebraically $T_\infty \Embbar(S^m,\R^n)^\Q$ beyond the range of~\eqref{eq:MC2}.
 On the other hand,  the equivalence~\eqref{eq:MC1} had been proved earlier by B.~Fresse
and the authors in~\cite[Theorem~1]{FTW}. In~\cite[Corollaries~5 and~8]{FTW}, we  similarly expressed $T_\infty \Embbar_\partial(\R^{n-1},\R^n)^\Q$
and $T_\infty \Embbar_\partial(\R^{n},\R^n)^\Q$ using the same, up to a degree one shift, graph-complex $\GC_n$,
the usual Kontsevich graph-complex of bald (no hairs) graphs endowed in both cases with the abelian $L_\infty$-structure.
 The reason we get a smaller
complex  for $n-m=1$ is the relative non-formality of the little discs operads in codimension one~\cite{TW}. 
As a consequence, we obtain \cite[equation~(14)]{FTW}
\[
T_\infty \Embbar_\partial(\R^{n},\R^n)^\Q\simeq \Omega T_\infty \Embbar_\partial(\R^{n-1},\R^n)^\Q.
\]
Note that one also has
\[
\Diff_\partial(D^n)\simeq\Omega\Emb_\partial(\R^{n-1},\R^n),
\]
see \cite[Appendix, Section~5, Proposition~5]{Cerf}, \cite[Proposition~5.3]{Budney2}, which  implies
\[
\Embbar_\partial(\R^n,\R^n):= \hofiber\left(\Diff_\partial(D^n) \to \Omega^n\SO(n)\right)\simeq \Omega\Embbar_\partial(\R^{n-1},\R^n).
\]
 This means that,
even though the Goodwillie-Weiss calculus is only applicable in codimensions $\geq 3$, it can still detect at least
rationally the codimension one versus codimension zero rigidity of embeddings.

\subsection*{Acknowledgment} The authors thank R.~Budney and  B.~Fresse  for communication.

\end{document}